\documentclass{amsart}
\usepackage{amsmath,amssymb,mathtools,amsthm}
\usepackage[utf8]{inputenc}
\usepackage{xcolor}
\usepackage{appendix}

\DeclareMathOperator{\rk}{rank}

\newcommand{\Z}{\mathbb{Z}}	
\newcommand{\R}{\mathbb{R}}	
\newcommand{\prob}{\mathbb{P}}	
\usepackage[english]{babel}
\usepackage{textcomp}
\usepackage{mathtools}
\usepackage{url}
\usepackage{todonotes}
\usepackage[T1]{fontenc}
\usepackage{pgfplots}
\pgfplotsset{compat=1.15}
\usepackage{mathrsfs}
\usetikzlibrary{arrows}
\usepackage{empheq}

\newcommand{\ds}{\displaystyle}
\newtheorem{theorem}{Theorem}[section]

\newtheorem{lemma}[theorem]{Lemma}

\newtheorem{prop}[theorem]{Proposition}

\newtheorem{conjecture}[theorem]{Conjecture}
\newtheorem*{thm:length2}{Theorem \ref{thm:length2}}
\newtheorem*{thm:length3}{Theorem \ref{thm:length3}}
\newtheorem*{thm:equalodd}{Theorem \ref{thm:equalodd}}
\newtheorem*{thm:equaleven}{Theorem \ref{thm:equaleven}}
\theoremstyle{definition}
\theoremstyle{remark}
\title[Random chain complexes of real vector spaces]{Random chain complexes of real vector spaces}
\author{Ayat Ababneh}
\author{Matthew Kahle}
\date{\today}

\begin{document}

\maketitle

\begin{abstract}
We introduce a natural class of models of random chain complexes of real vector spaces that some classical ensembles of random matrices, the length $1$ case. We are interested here in the homological properties of these random complexes. For chain complexes of length $1$ or $2$, we characterize the Betti numbers almost surely, in terms of the dimensions of the vector spaces. We further examine complexes of length $3$ with some constraints on dimensions, as well as complexes of arbitrary finite length in which all vector spaces have equal dimension. Across all these settings, we show that the sum of the Betti numbers is almost surely as small as possible, attaining a trivial lower bound $|\chi|$ dictated by the dimensions of the underlying vector spaces and the Euler formula. These results suggest an underlying algebraic heuristic for a phenomenon frequently observed in stochastic topology, that nontrivial homology rarely appears unless forced to.
\end{abstract}
\section{Introduction}

Stochastic topology has been an active area of research for the past 15--20 years. Various models have been studied, including random manifolds \cite{DunThurston}, random simplicial complexes \cite{BK18, Kahlesurvey18}, random groups \cite{Ollivier05}, random knots \cite{EZ17,EZHLN18}, etc. Algebraic topological invariants of such random topological objects have been studied extensively, and qualitatively similar phenomena have been observed in disparate settings.

Dunfield and W.\ Thurston \cite{DunThurston} introduced a natural model of random $3$-manifold $M$, and showed that with high probability $H_1(M, \R) =0$. Similar homology-vanishing phenomena have been observed in random simplicial complexes, over $\R$ and over finite fields. In these settings, there is often a trivial combinatorial obstruction to cohomology vanishing. Several theorems in the area show that as soon as some combinatorial obstruction disappears, cohomology vanishes \cite{Bobrowski22,Kahle14, KPR21,KP16, LM2006,MW2009,NP23}. It is also sometimes observed that in a certain parameter regime the number of $k$-faces is much more than the number of $k-1$-faces or $k+1$-faces. In such cases, $\beta_k$ is forced to be nonzero for dimensional reasons \cite{CF17,Fowler19,Kahle2009}.

We study here homology at the chain complex level. Random chain complexes of vector spaces over finite fields were studied earlier by Ginzburg and Pasechnik \cite{GP17} and by Catanzaro and Zabka \cite{CZ21}. We work over the reals, which has a somewhat different flavor. In the finite-field setting, one naturally considers asymptotics where some parameter, either the order of the finite field or the length of the chain complex, tends to infinity. One might be interested in events that occur ``asymptotically almost surely'', where the probability tends to one. Here we have results ``almost surely'' where the probability equals one. It is also slightly more subtle measure-theoretically to define random chain complexes of real vector spaces, since counting measure is not available. The following section makes this precise.\\

In the following, we consider bounded chain complexes of real vector spaces $(A_*, d_*)$,
\[0\leftarrow A_0\xleftarrow{d_1}A_1 \xleftarrow{d_2} \dots  \xleftarrow{d_n} A_n \leftarrow 0.\]
We assume the dimensions of the vector spaces are given $a_i = \dim A_i$, and the \emph{Euler characteristic} of the chain is then defined by
\[ \chi = \sum_{i=0}^n (-1)^i a_i.\]

Recall that if the ranks of the maps are given by $r_i = \rk d_i$, then we can define the \emph{Betti numbers} by
\[ \beta_i = - r_i + a_i  - r_{i+1}.\]
Equivalently, $\beta_i = \dim H_i(A_*)$.  The well-known Euler formula states that
\[ \chi = \sum_{i=0}^n (-1)^i \beta_i.\]
Taking absolute values of both sides and applying the triangle inequality, we have that
\[ |\chi| \le \sum_{i=0}^n \beta_i.\]
If equality occurs, we say total homology of the chain complex is as small as possible. \\

The precise definition for random chain complexes of real vector spaces will be given in the next section. We state our main results first.

\begin{thm:length2}
For a random chain complex of length $2$, 
\[0\leftarrow A_0\xleftarrow{d_1}A_1\xleftarrow{d_2}A_2\leftarrow 0,\]
almost surely 
\[(\beta_0,\beta_1,\beta_2)=\left\{\begin{array}{lll}
(a_0-a_1~, ~0~, ~a_2),&  & \mbox{if } a_0\geq a_1+a_2,\\
(a_0~,~ 0~,~ a_1-a_2),& & \mbox{if } a_2\geq a_0+a_1,\\
(0~,~ a_1-a_0-a_2~,~ 0), &&  \mbox{if }a_1\geq a_0+a_2,\\
(\chi/2~,~ 0 ~,~ \chi/2), && \mbox{if }a_2-a_1\le a_0 \le a_1+a_2, \,  \\
&& a_1 \le a_0+a_2,  \mbox{ and } \chi \mbox{ is even,}\\
\left((\chi\pm 1 )/2~,~ 0 ~,~ (\chi \mp 1)/2\right), && \mbox{if }a_2-a_1\le a_0 \le a_1+a_2,\\
&& a_1 \le a_0+a_2, \mbox{ and } \chi \mbox{ is odd.}\\
\end{array}
\right.
\]
It follows that almost surely
\[\sum_{i=0}^2 \beta_i = |\chi|.\]
\end{thm:length2}

\vspace{0.5in}

For a random chain complex of length $2$, total homology is almost surely as small as possible. We cannot expect this to hold for a random chain complex of length $3$ without further assumptions. For example, if $(a_0,a_1, a_2, a_3) = (2,1,1,2)$, then for every chain complex we have nontrivial homology $\beta_0 \ge 1$ and $\beta_3 \ge 1$, but $|\chi| = 0$. The problem with this example is that we are forcing homology to appear by the dimensions of the vector spaces. So a reasonable extra assumption is that $a_i+a_{i+2}\geq a_{i+1}$ for every $i$, since otherwise $\beta_{i+1}$ is forced to be positive. With those extra assumptions, we show that total homology is indeed almost surely as small as possible. 

\begin{thm:length3}
For a random chain complex of length $3$, 
\[0\leftarrow A_0\xleftarrow{d_1} A_1\xleftarrow{d_2} A_2\xleftarrow{d_3} A_3\leftarrow 0,\]
where $a_i+a_{i+2}\geq a_{i+1}$ for $i=-1,0,1,2$, almost surely $\sum \beta_i=|\chi|$.
\end{thm:length3}

Next, we study random bounded chain complexes of real vector spaces of arbitrary length, where the vector spaces all have equal dimension. First, we consider the case of $n$ odd. In this case, the random chain complex is almost surely an exact sequence.  

\begin{thm:equalodd}
Let $n$ be odd, and let $m$ be a positive integer. For a random chain complex
$0\leftarrow A_0 \xleftarrow{d_1} A_1\xleftarrow{d_2} \cdots\xleftarrow{d_n} A_n\leftarrow 0$ where $a_i=m$ for every $i$, almost surely $\beta_i=0$ for $i=0,1, \dots,n.$ \\
\end{thm:equalodd}

Finally, we consider the case of equal dimensions and $n$ even.

\begin{thm:equaleven}
Let $n$ be even, and let $m$ be a positive integer. For a random chain complex
$0\leftarrow A_0 \xleftarrow{d_1} A_1\xleftarrow{d_2} \cdots\xleftarrow{d_n} A_n\leftarrow 0$ where $a_i=m$ for every $i$, almost surely $\sum \beta_i=|\chi|$.\\
\end{thm:equaleven}

In the $n$ even case, we show further that all the odd Betti numbers vanish, and that either $\beta_{2k}=\left\lfloor\frac{m}{n/2+1}\right\rfloor$ or $\beta_{2k}=\left\lceil\frac{m}{n/2+1}\right\rceil$ for  $k=0,1,2,\cdots,\frac{n}{2}.$ Note that $|\chi| = m$, so this is saying that the total homology is spread out as evenly as possible. If $m / (n/2 + 1)$ is not an integer, then this can give many possibilities for the vector of Betti numbers $(\beta_0, \dots, \beta_m)$. This is in contrast to many of the examples above, where the vector of Betti numbers is concentrated on a single vector.

\subsection*{Acknowledgements} The work of the first author was partially supported by a research grant from the University of Jordan. The second author thanks Elliot Paquette and Kenny Easwaran for helpful discussions about conditioning on an event of probability zero and Dave Anderson for helpful discussions about varieties of chain complexes.

\section{The model}
In this section we introduce a natural class of models of random chain complexes over $\R$ and compute a few of its basic features.

\subsection{Conditioning on an event of probability zero}

Conditional probability $\prob(B \mid A) $ is usually defined by
\[
\prob(B \mid A) := \frac{ \prob( A \mbox{ and } B)}{\prob(A)}.
\]
However intuitive, this definition is of limited use if $\prob(A) = 0$. A lot has been written about conditioning on an event of probability zero, both mathematically \cite{BD75} and philosophically \cite{Easwaran2019}. 

Let $f : \R^N \to\R$ be a bounded, measurable, continuous probability density function. For example, let $f$ be a standard multivariate normal distribution. Assume that Hausdorff dimension $\dim_H X =k$, and that $X$ and $Y$ are rectifiable subsets of $\R^N$. In the following, $X$ will always be an affine algebraic variety in $\R^n$ and $Y$ a subvariety, so this will certainly hold. 

We define
\[
\prob(Y \mid X)=  \frac{\int_{Y} f(x) d \square^k(x)}{\int_{X} f(x) d \square^k(x)},
\]
where $d\square^k(p)$ is the $k$-dimensional Hausdorff (surface) measure.

The simple but useful fact that we apply throughout the rest of the paper is that if $\dim_H(Y) < \dim_H(X)$, then the numerator is zero and $\prob(Y \mid X)=0$.

In the following, $\R^N$ will be the space of all sequences of linear maps
\[
0\leftarrow A_0\xleftarrow{d_1} A_1\xleftarrow{d_2} \cdots \xleftarrow{d_{n}}  A_n\leftarrow 0,
\]
$X = M(\vec{a}) \subseteq R^N$ will be the variety of chain complexes of vector spaces dimensions $\vec{a}$, and $Y = P(\vec{a},\vec{r})$ the subvariety of $M(\vec{a})$ where the ranks of the maps are given by $\vec{r} = (r_1, \dots, r_n)$.


An equivalent definition is the following. For a subset $A \subseteq \R^N$, let $A(\epsilon)$ denote the $\epsilon$-thickening of $A$. That is, 
\[ A(\epsilon) := \left\{ x \in \R^N \mid d(x ,A) < \epsilon \right\}.\]
Then we can equivalently define 
\[
\prob(Y \mid X)= \lim_{\epsilon \to 0^+}\frac{\int_{Y(\epsilon)} f(x) dx}{\int_{X(\epsilon)} f(x) dx},
\]
where $dx$ represents Lebesgue measure. 

The fact that these definitions are equivalent is standard, and is closely related to the notion of disintegration of measures. We do not actually need $f$ to be continuous --- it is sufficient, for example, for $f$ to be bounded and Borel measurable \cite{Federer69,DHS13}. We could even consider a ground probability measure on $\R^N$ without a density function, as long as it is absolutely continuous with respect to Lebesgue measure.

So in some strong sense, all our main results hold quite generally, simply for dimensional reasons. At the same time, we believe these definitions are quite natural, and in fact this allows one to define random chain complexes of real vector spaces in terms of classical random matrix ensembles. The Gaussian example mentioned above is an important special case, generalizing the well-studied random-matrix model with i.i.d.\ Gaussian entries to longer random complexes. 

\subsection{Random chain complexes}

Given a vector of non-negative integers $\vec{a} = (a_0, \dots, a_n)$, for each $i = 0, 1, \dots, n$, let $A_i = \R^{a_i}$ be a real vector space of dimension $a_i$. We define random chain complexes
\[
0\leftarrow A_0\xleftarrow{d_1} A_1\xleftarrow{d_2} \cdots \xleftarrow{d_{n}}  A_n\leftarrow 0.
\]

Let $M(\vec{a})$ denote the space of all chain complexes. We can think of $M( \vec{a})$ as an affine algebraic variety in $\mathbb{R}^N$, where $N=a_0a_1 + \dots +a_{n-1}a_n$, in the following sense. We identify each boundary map $d_i$ as given by an $a_{i+1} \times a_i$ matrix. The total number of entries of the matrices is $N$. For every $i$, the condition $d_{i+1} d_i  = 0$ is equivalent to the vanishing of a set of degree-two polynomials in the entries of the matrices $d_i$ and $d_{i+1}$. 

Now, given another vector of non-negative integers $\vec{r} =(r_1, r_1, \dots, r_{n})$, let $P(\vec{a}, \vec{r})$ denote the subspace of $M(\vec{a})$ where for every $i=1, \dots, n$ we have $\mbox{rank } d_i = r_i$.

For some choices of $\vec{a}$ and $\vec{r}$, the subspace $P(\vec{a},\vec{r})$ is empty --- i.e.\ there do not exist chain complexes with the given dimensions and ranks. The following proposition establishes necessary and sufficient conditions for existence. This condition must be well known, but we include the statement here for the sake of completeness. 

\begin{prop}\label{prop:existence2}
There exists a chain complex 
\[ 0 \leftarrow A_0 \xleftarrow{d_1} A_1 \xleftarrow{d_2}\cdots\xleftarrow{d_{n}} A_n \leftarrow 0\]
with $\dim A_i = a_i$ for $i = 0, \dots, n$ and $\mbox{rank }d_i=r_i$ for $i=1, \dots, n$, if and only if $r_i + r_{i+1} \le a_i$ for every $i = 0, \dots, n$.
\end{prop}

The proof is straightforward so we omit it.

\begin{prop} \label{prop:maindim}
Let $\vec{a}=(a_0, a_1, \dots, a_n)$ and $\vec{r}=(r_1, r_2, \dots, r_n)$ be vectors of nonnegative integers. Whenever a chain complex is possible with vector space dimensions given by $\vec{a}$ and ranks of maps given by $\vec{r}$, the dimension of the subvariety of chain complexes of real vector spaces
\[
0\leftarrow A_0\xleftarrow{d_1}A_1\xleftarrow{d_2}A_2\leftarrow \cdots\xleftarrow{d_{n}}A_{n}\leftarrow 0
\]
with ranks $\vec{r}$ is given by the formula 
\begin{align*}
d(\vec{a},\vec{r})=\sum_{i=1}^{n}r_i(a_{i}+a_{i-1}-r_{i-1}-r_{i}).
\end{align*}
\end{prop}

This formula is well known to algebraic geometers --- see, for example, De Concini and Strickland's paper on varieties of complexes \cite{DCS81}. The varieties studied there have the condition $\mbox{rank } d_i \le r_i$ rather than $\mbox{rank } d_i = r_i$, but the latter is an open subvariety of the former so they have the same dimension. 
We will see in the following that for many choices of rank vector $\vec{r}$, $\dim P(\vec{a},\vec{r}) < \dim M(\vec{a})$. By the earlier discussion, $\dim P(\vec{a},\vec{r}) < \dim M(\vec{a})$, then the probability that a random chain complex with vector-space dimensions $\vec{a}$ has rank vector $\vec{r}$ is zero. For a start, we note the following useful lemma.

\begin{lemma}[Monotonocity Lemma]\label{lemma:increasing}
For the random chain complex 
$$0\leftarrow A_0\leftarrow A_1\leftarrow\cdots\leftarrow A_n\leftarrow 0.$$
If $\vec{r}=(r_1,\cdots,r_i+1,\cdots,r_{n})$ is in the allowable region, then
$$d(\vec{a},(r_1,\cdots,r_i+1,\cdots,r_{n}))> d(\vec{a},(r_1,\cdots,r_i,\cdots,r_{n})).$$
That is, the dimension of the space of chain complexes is strictly increasing in each rank.
\end{lemma}
\begin{proof}
Let $\vec{\rho_1}=(r_1,\cdots,r_i+1,\cdots,r_{n})$ and $\vec{\rho_2}=(r_1,\cdots,r_i,\cdots,r_{n}).$
We have
$$
\begin{array}{lll}
d(\vec{a},\vec{\rho_1})-d(\vec{a},\vec{\rho_2})&=&(r_{i}+1)(a_{i-1}+a_i-(r_i+1)-r_{i-1})\\
&&+(r_{i+1})(a_i+a_{i+1}-r_{i+1}-(r_i+1))\\
&&-r_{i+1}(a_{i+1}+a_{i}-r_{i}-r_{i+1})-r_i(a_{i}+a_{i-1}-r_{i-1}-r_{i})\\
&=& -r_{i+1}+a_i+a_{i-1}-r_{i-1}-(r_i+1)-r_i\\
&=&a_{i-1}-(r_{i-1}+r_i)+a_{i}-(r_i+r_{i+1})-1\\
&\geq& 1.
\end{array}
$$
The last inequality follows from the assumption that the rank vector $\vec{\rho_1}$ is in the allowable region, which means $r_{i-1}+r_i+1\leq a_{i-1}$ and $r_i+1+r_{i+1}\leq a_{i}.$
\end{proof}

\section{Random chain complexes of length $1$, $2$, and $3$}

We now focus on the special cases of lengths 1, 2, and 3.

\subsection{Chain complexes of length 1}

\begin{theorem}\label{prop:2vectors}
For a random chain complex $0\leftarrow A_0\xleftarrow{d_1} A_1\leftarrow 0$, we have
$$
(\beta_0,\beta_1)=\ds\left\{\begin{array}{lll}
(0,a_1-a_0),&& a_0\leq a_1,\\
(a_0-a_1,0), && a_1<a_0.
\end{array}
\right.
$$
\end{theorem}

\begin{proof}
As a special case of the above formula, $d(\vec{a},\vec{r})=r_1(a_0+a_1-r)=-r_1^2+(a_0+a_1)r_1$, where $r_1$ is the rank of the map $d_1.$ To determine which rank vectors have positive probability, we maximize $d$ subject to the constraint $0\leq r_1 \leq \min\{a_0,a_1\}.$ By the monotonicity lemma, $d$ is increasing in $r_1$ and has a maximum when $r_1=\min\{a_0,a_1\}.$ Therefore, the Betti numbers are $\beta_0=a_0-\min\{a_0,a_1\}$ and $\beta_1=a_1-\min\{a_0,a_1\},$ as desired.  
\end{proof}

\subsection{Chain complexes of length 2}
The homology of random chain complexes of length $2$ is characterized by the following theorem.


\begin{theorem} \label{thm:length2}
For a random chain complex of length $2$, 
\[0\leftarrow A_0\xleftarrow{d_1}A_1\xleftarrow{d_2}A_2\leftarrow 0,\]
almost surely 
\[(\beta_0,\beta_1,\beta_2)=\left\{\begin{array}{lll}
(a_0-a_1~, ~0~, ~a_2),&  & \mbox{if } a_0\geq a_1+a_2,\\
(a_0~,~ 0~,~ a_1-a_2),& & \mbox{if } a_2\geq a_0+a_1,\\
(0~,~ a_1-a_0-a_2~,~ 0), &&  \mbox{if }a_1\geq a_0+a_2,\\
(\chi/2~,~ 0 ~,~ \chi/2), && \mbox{if }a_2-a_1\le a_0 \le a_1+a_2, \,  \\
&& a_1 \le a_0+a_2,  \mbox{ and } \chi \mbox{ is even,}\\
\left((\chi\pm 1 )/2~,~ 0 ~,~ (\chi \mp 1)/2\right), && \mbox{if }a_2-a_1\le a_0 \le a_1+a_2,\\
&& a_1 \le a_0+a_2, \mbox{ and } \chi \mbox{ is odd.}\\
\end{array}
\right.
\]
It follows that almost surely
\[\sum_{i=0}^2 \beta_i = |\chi|.\]
\end{theorem}

\medskip

In the final case, there are two distinct choices of rank vectors that maximize the dimension $d$, and each occurs with positive probability. For example, when $a_0=a_1=a_2=m$ where $m$ is odd, the Betti number vectors $(\beta_0, \beta_1, \beta_2)$ are
\[
\bigl((m-1)/2,\,0,\,(m+1)/2\bigr)
\quad\text{and}\quad
\bigl((m+1)/2,\,0,\,(m-1)/2\bigr).
\]

Boundary cases where equality holds in the defining inequalities are compatible with the stated formulas. Indeed, if $a_0=a_1+a_2$, then
$a_0-a_1=(a_0-a_1+a_2)/2=a_2$; if $a_2=a_0+a_1$, then
$a_2-a_1=(a_0-a_1+a_2)/2=a_0$; and if $a_1=a_0+a_2$, then
$a_1-a_0-a_2=(a_0-a_1+a_2)/2=0$.

Theorem \ref{thm:length2} covers all triples of nonnegative integers $(a_0,a_1,a_2)$. Either one of the inequalities
$a_0\ge a_1+a_2$, $a_1\ge a_0+a_2$, or $a_2\ge a_0+a_1$ holds, or else all three fail, in which case
$a_2-a_1<a_0<a_1+a_2$.

\begin{proof}
We want to find the value of $\vec{r}=(r_1,r_2)$ that maximize the dimension $d(\vec{a},\vec{r})$, subject to the constraints
\begin{align*} \label{eq: constraints}
r_1 & \leq  \min\{a_0,a_1\},\\
r_2 &\leq \min\{a_1,a_2\}, \mbox{ and }\\
   r_1+r_2  & \leq a_1. 
\end{align*}
From Proposition \ref{prop:maindim},
\[ d(\vec{a},\vec{r})=r_1(a_0+a_1)+r_2(a_1+a_3)-r_1^2-r_2^2-r_1r_2.\]

By Lemma~\ref{lemma:increasing}, within the region of feasible rank vectors, $d(\vec{a},\vec{r})$ increases when moving $\vec{r}$ right or upward in the lattice. Hence, the dimension $d$ is maximized at lattice points in the upper-right part of the region. We therefore consider two cases.

\begin{enumerate}
    \item First, suppose that $a_1\geq a_0+a_2$. The feasible region defined by the constraints is the shaded rectangle in Figure~\ref{fig:square region}, whose upper-right corner is the point $P=(a_0,a_2).$ The dimension is therefore maximized at $r_1=a_0, r_2=a_2$, yielding Betti number $\beta_1=a_1-a_0-a_2.$
    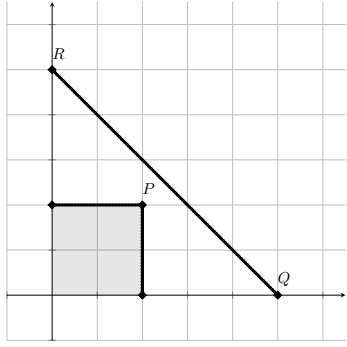
\begin{figure}
    \centering
   \definecolor{qqqqff}{rgb}{0,0,1}
\definecolor{ududff}{rgb}{0.30196078431372547,0.30196078431372547,1}
\definecolor{xdxdff}{rgb}{0.49019607843137253,0.49019607843137253,1}
\tikz \node [scale=0.6, inner sep=0] {
\begin{tikzpicture}[line cap=round,line join=round,>=triangle 45,x=1cm,y=1cm]
\begin{axis}[
x=1cm,y=1cm,
axis lines=middle,
ymajorgrids=true,
xmajorgrids=true,
xmin=-1,
xmax=6.5,
ymin=-1,
ymax=6.5,
xticklabels={,,},
yticklabels={,,},]
\clip(-1,-1) rectangle (6.5,6.5);
\fill[line width=2pt,dash pattern=on 1pt off 1pt,color=black,fill=black,fill opacity=0.1] (0,0) -- (2,0) -- (2,2) -- (0,2) -- cycle;
\draw [line width=2pt,dash pattern=on 1pt off 1pt] (0,5)-- (5,0);
\draw [line width=2pt,dash pattern=on 1pt off 1pt] (2,2)-- (2,0);
\draw [line width=2pt,dash pattern=on 1pt off 1pt,color=black] (2,0)-- (2,2);
\draw [line width=2pt,dash pattern=on 1pt off 1pt,color=black] (2,2)-- (0,2);
\begin{scriptsize}
\draw [fill=black] (0,5) ++(-2.5pt,0 pt) -- ++(2.5pt,2.5pt)--++(2.5pt,-2.5pt)--++(-2.5pt,-2.5pt)--++(-2.5pt,2.5pt);
\draw[color=black] (0.14198613176948002,5.358617787660826) node {$R$};
\draw [fill=black] (5,0) ++(-2.5pt,0 pt) -- ++(2.5pt,2.5pt)--++(2.5pt,-2.5pt)--++(-2.5pt,-2.5pt)--++(-2.5pt,2.5pt);
\draw[color=black] (5.1397402582365,0.3608636611938056) node {$Q$};
\draw [fill=black] (2,2) ++(-2.5pt,0 pt) -- ++(2.5pt,2.5pt)--++(2.5pt,-2.5pt)--++(-2.5pt,-2.5pt)--++(-2.5pt,2.5pt);
\draw[color=black] (2.137676346092488,2.3736110568358137) node {$P$};
\draw [fill=black] (2,0) ++(-2.5pt,0 pt) -- ++(2.5pt,2.5pt)--++(2.5pt,-2.5pt)--++(-2.5pt,-2.5pt)--++(-2.5pt,2.5pt);
\draw [fill=black] (0,2) ++(-2.5pt,0 pt) -- ++(2.5pt,2.5pt)--++(2.5pt,-2.5pt)--++(-2.5pt,-2.5pt)--++(-2.5pt,2.5pt);
\end{scriptsize}
\end{axis}
\end{tikzpicture}
};
 \caption{The feasible region for $(r_1,r_2)$ in the case $a_1\geq a_0+a_2$. The diagonal line is the line $r_1+r_2 = a_1$, and the point $P=(a_0,a_2)$.}
    \label{fig:square region}
\end{figure}
    \item Now, suppose that $a_1< a_0+a_2$. In this case, any points $(r_1,r_2)$ that maximize the dimension will be on the line $r_1+r_2=a_1$. See Figure \ref{fig:trapezoid region}. 

         \begin{figure}
    \centering
        \definecolor{ccccff}{rgb}{0.8,0.8,1}
\definecolor{xdxdff}{rgb}{0.49019607843137253,0.49019607843137253,1}
\tikz \node [scale=0.6, inner sep=0] {
\begin{tikzpicture}[line cap=round,line join=round,>=triangle 45,x=1cm,y=1cm]
\begin{axis}[
x=1cm,y=1cm,
axis lines=middle,
ymajorgrids=true,
xmajorgrids=true,
xmin=-1,
xmax=6.5,
ymin=-1,
ymax=6.5,
xticklabels={,,},
yticklabels={,,},]
\clip(-1,-1) rectangle (6.5,6.5);
\fill[line width=2pt,color=black,fill=black,fill opacity=0.1] (0,0) -- (3,0) -- (3,2) -- (0,5) -- cycle;
\draw [line width=2pt,dash pattern=on 1pt off 1pt] (0,5)-- (5,0);
\draw [line width=2pt,color=black] (0,0)-- (3,0);
\draw [line width=2pt,color=black] (3,0)-- (3,2);
\draw [line width=2pt,color=black] (3,2)-- (0,5);
\draw [line width=2pt,color=black] (0,5)-- (0,0);
\begin{scriptsize}
\draw [fill=black] (0,5) ++(-2.5pt,0 pt) -- ++(2.5pt,2.5pt)--++(2.5pt,-2.5pt)--++(-2.5pt,-2.5pt)--++(-2.5pt,2.5pt);
\draw [fill=black] (2,3) ++(-2.5pt,0 pt) -- ++(2.5pt,2.5pt)--++(2.5pt,-2.5pt)--++(-2.5pt,-2.5pt)--++(-2.5pt,2.5pt);
\draw [fill=black] (3,2) ++(-2.5pt,0 pt) -- ++(2.5pt,2.5pt)--++(2.5pt,-2.5pt)--++(-2.5pt,-2.5pt)--++(-2.5pt,2.5pt);
\draw [fill=black] (1,4) ++(-2.5pt,0 pt) -- ++(2.5pt,2.5pt)--++(2.5pt,-2.5pt)--++(-2.5pt,-2.5pt)--++(-2.5pt,2.5pt);  
\end{scriptsize}
\end{axis}
\end{tikzpicture}
};
\caption{The feasible region for $(r_1,r_2)$ in the case $a_1< a_0+a_2$ and $a_0<a_1\leq a_2$. In this case, any points $(r_1,r_2)$ that maximize dimension must be on the line $r_1+r_2=a_1$, by the Monotonicity Lemma.}
    \label{fig:trapezoid region}
\end{figure}
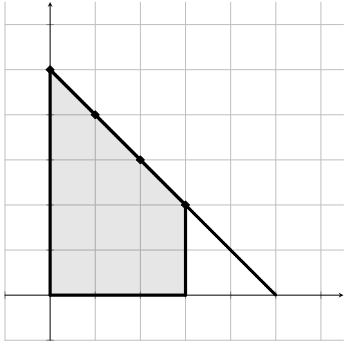

    \end{enumerate}  

To sum up, we proved that $r_1+r_2=\min\{a_1, a_0+a_2\}$ which gives $\beta_1=\max\{0, a_1-a_0-a_2\}$ as desired.

From the above, we know that when $a_1\geq a_0+a_2$, the ranks have values $(r_1,r_2)=(a_0,a_2)$, which give the Betti numbers $\beta_0=0,~\beta_2=0.$ When $a_1<a_0+a_2,$ we saw that the ranks satisfy $r_1+r_2=a_1,$  so
\begin{align*}
d &=r_1(a_0+a_1)+(a_1-r_1)(a_1+a_2)-r_1^2-(a_1-r_1)^2-r_1(a_1-r_1)\\
&=-r_1^2+r_1(a_0+a_1-a_2)+a_1a_2.\\
\end{align*}

The function $d$, now just a function of the single variable $r_1$, is maximized at $r_1=(a_0+a_1-a_2)/2.$ Since $r_2=a_1-r_1$, we have $r_2=(-a_0+a_1+a_2)/2$ and these ranks are in the allowable region as long as $a_2-a_1\leq a_0\leq a_1+a_2.$ So we have two cases:
\begin{enumerate}
    \item\label{item:1,1} If $a_2-a_1\leq a_0\leq a_1+a_2,$ then the rank vector is feasible and we have two subcases.
    \begin{enumerate}
        \item If $a_0-a_1+a_2$ is even, then testing the sign of $d^{'}$ gives that the dimension is maximum when $(r_1,r_2)=\left((a_0+a_1-a_2)/2,(a_1+a_2-a_0)/2\right).$
        
        \item If $a_0-a_1+a_2$ is odd, then the maximum value of the dimension is either when $r_1=(a_0+a_1-a_2-1)/2$ or when $r_1=(a_0+a_1-a_2+1)/2$ as long as these ranks are in the allowable region. Calculating the dimension at these two values of the ranks, we get the same answer, so we have two values of $\vec{r}$ that maximize $d$.
        
    \end{enumerate}
    \item If the inequality $a_2-a_1\leq a_0\leq a_1+a_2$ is not satisfied, that is if $a_0>a_1+a_2$ or if $a_0<a_2-a_1.$ Then the rank that maximizes the dimension will be an endpoint and in all possible orders of the three dimensions the rank will be either $(r_1,r_2)=(a_1,0)$ or $(r_1,r_2)=(0,a_1).$
        \end{enumerate}
\end{proof}

\subsection{Chain complexes of length 3}

Having established a complete characterization of homology for chains of lengths $1$ or $2$, we now study chains of length $3$, with some restrictions on the dimensions of the vector spaces.
\begin{theorem} \label{thm:length3}
For a random chain complex of length $3$, 
\[0\leftarrow A_0\xleftarrow{d_1} A_1\xleftarrow{d_2} A_2\xleftarrow{d_3} A_3\leftarrow 0,\]
where $a_i+a_{i+2}\geq a_{i+1}$ for $i=-1,0,1,2$, almost surely $\sum \beta_i=|\chi|$.
\end{theorem}
\begin{proof}
 Let $\vec{r}=(r_1,r_2,r_3)$ be a rank vector that maximizes the dimension of the space of chains and recall that $\ds\sum_{i=0}^{3}\beta_i=\ds\sum_{i=0}^{3}a_i-2r_i$. To prove the statement of the proposition, we will consider three cases.
  
  {\bf{Case 1}}: $a_0+a_2> a_1+a_3.$
Then $|\chi|=a_0-a_1+a_2-a_3$ and $r_1+r_2+r_3=a_1+a_3$.
If $r_3<a_3$ and $r_2+r_3<a_2$, then $(r_1,r_2,r_3+1)$ is allowable and increases $d$, contradicting maximality. Thus $r_3<a_3$ implies $r_2+r_3=a_2$. If moreover $r_1=a_0$, then $a_0+a_2\le a_1+a_3$, a contradiction; hence $r_1<a_0$, and $(r_1+1,r_2-1,r_3)$ again increases $d$. Therefore $r_3=a_3$. Since $r_1+r_2+r_3=a_1+a_3$, we have $r_1+r_2\le a_1$, with equality forced by the same exchange argument. Thus $r_1+r_2=a_1$.

\smallskip

   {\bf{Case 2}}: $a_1+a_3>a_0+a_2.$
Then $|\chi|=-a_0+a_1-a_2+a_3$ and $r_1+r_2+r_3=a_0+a_2$.
If $r_1<a_0$ and $r_1+r_2<a_1$, then $(r_1+1,r_2,r_3)$ increases $d$.
If $r_1<a_0$ and $r_1+r_2=a_1$, then necessarily $r_3<a_3$, and $(r_1,r_2-1,r_3+1)$ increases $d$. Hence $r_1=a_0$, and by the argument of the previous case, $r_2+r_3=a_2$.

\smallskip

    {\bf{Case 3}}: $a_0+a_2=a_1+a_3.$
Then $|\chi|=0$ and $r_1+r_2+r_3=a_0+a_2$. If $r_1<a_0$, then either $(r_1+1,r_2,r_3)$ or $(r_1+1,r_2-1,r_3+1)$ is allowable and increases $d$, a contradiction; hence $r_1=a_0$. By symmetry, $r_3=a_3$, and therefore
\[
r_2=(a_0+a_2)-(a_0+a_3)=a_2-a_3=a_1-a_0.
\]
\end{proof}


\section{Random complexes with vector spaces of equal dimension}

In this section, we study random chain complexes of arbitrary length where the vector spaces all have equal dimensions. First, we study random chain complexes of odd length. They are almost surely exact sequences.

\begin{theorem} \label{thm:equalodd}
Let $n$ be odd, and let $m$ be a positive integer. For a random chain complex
$0\leftarrow A_0 \xleftarrow{d_1} A_1\xleftarrow{d_2} \cdots\xleftarrow{d_n} A_n\leftarrow 0$ where $a_i=m$ for every $i$, almost surely $\beta_i=0$ for $i=0,1, \dots,n.$ \\
\end{theorem}
\begin{proof}
For a random chain complex of length $n$, where $n$ is odd, all of the same dimension, $m$, the dimension of the space of chain complexes is given by the formula 
$$d(\vec{a},\vec{r})=\sum_{i=1}^{n}r_i(2m-r_{i-1}-r_{i}).$$
Calculating the Hessian for this function we get
$$
 \mathbf {H} _{d}={\begin{bmatrix}-2&-1&0&\ldots &0\\-1&-2&-1&\ddots &\vdots \\0&\ddots &\ddots &\ddots &0\\
 \vdots &\ddots&-1&-2&-1\\0&\ldots &0&-1&-2\end{bmatrix}},
$$
which is a constant negative-definite matrix. Therefore, the function $d$ is convex, which in turn implies that any local maximum for $d$ is a global maximum. Now, to find a local maximum we can solve the following system of equations
$$
\begin{array}{lll}
\ds\frac{\partial d}{\partial r_1}&=&2m-2r_1-r_2=0,\\\\
\ds\frac{\partial d}{\partial r_2}&=&2m-r_1-2r_2-r_3=0,\\
&~\vdots&\\
\ds\frac{\partial d}{\partial r_{n-1}}&=&2m-r_{n-2}-2r_{n-1}-r_n=0,\\\\
\ds\frac{\partial d}{\partial r_{n}}&=&2m-2r_n-r_{n-1}.
\end{array}
$$

Solving this linear system, we obtain a unique solution $\vec{r}=(m,0,m,\cdots,m,0,m)$. This in a global maximum for $d$ as a function $\R^n \to \R$. It is also an integer solution in the feasible region. Therefore $\vec{r}=(m,0,m,\cdots,m,0,m)$ almost surely. \\
\end{proof}

Now we consider even length. The approach above based on convexity does not quite work. We omit the details, but the unique vector $\vec{r}$ that maximizes $d$ is not an integer vector and is not even in the feasible region. \\

\begin{theorem} \label{thm:equaleven}
Let $n$ be even, and let $m$ be a positive integer. For a random chain complex
$0\leftarrow A_0 \xleftarrow{d_1} A_1\xleftarrow{d_2} \cdots\xleftarrow{d_n} A_n\leftarrow 0$ where $a_i=m$ for every $i$, almost surely $\sum_{i=0}^n \beta_i=|\chi|$. We also have that almost surely $\beta_i = 0$ for $i=1, 3, \dots, n-1$.\\
\end{theorem}

\begin{proof}
Recall that the eligible rank vectors $\vec{r}=(r_1,\dots, r_n)$ have non-negative integer entries and satisfy 
$r_i + r_{i+1} \le m $ for $i = 0, \dots n$.

We show in the following that for any $\vec{r}$ that maximizes $d$ in the feasible region, $\sum_{i=1}^n \beta_i = m$. Since
$$
\sum_{i=0}^n \beta_i = (n+1)m - 2 \sum_{i=1}^n r_i,
$$
this is equivalent to showing that 
\[\sum_{i=1}^n r_i= nm / 2.\]

Suppose that $\vec{r}=(r_1,\dots, r_n)$ is a vector of ranks that maximizes $d$. 
Define

\[ S := \left\{ i \mid r_i + r_{i+1} = m \right\}.\]
By the Monotonicity Lemma \ref{lemma:increasing}, for every $i=1, 2, \dots, n-1$, we have that $S \cap \{ i, i +1 \} \neq \emptyset$. Indeed, if $i \notin S$ and $i+1 \notin S$, then one could increase $r_{i+1}$ and still be in the feasible region, strictly increasing $d$. This would contradict the assumption that $d(\vec{r})$ is maximal.

We will see next that in fact, $2k+1 \in S$ for every $k=0,1,  \dots, (n/2) -1$. We proceed by induction on $k$. First, the case $k=0$ is clear. By the above, we either have $r_1 =m$ or $r_1 + r_2 = m$. But if $r_1=m$, the condition that $r_1 + r_2 \le m$ forces $r_2=0$ and $r_1+r_2=m$ anyway.

So now suppose by a strong induction hypothesis that $\{ 1, 3, \dots 2k-1 \} \subseteq S$, and we will show that $2k+1 \in S$. If $2k+1 \notin S$, then immediately we must have $2k \in S$ and $2k+2 \in S$.

Let $b$ be the smallest number such that the entire interval of integers $\{ b, b+1, \dotsm, 2k \} \subseteq S$. Clearly $b \le 2k-1$. Moreover, since $\{ 1, 3, \dots 2k-1 \} \subseteq S$, $b$ must be odd. We have the equalities $r_b + r_{b+1} = m$,  $r_{b+1} + r_{b+2} = m$, \dots, $r_{2k} + r_{2k+1}  = m$. Let us write simply $r$ for the ranks of the odd maps, i.e. $r:=r_b, r_{b+2}, \dots, r_{2k+1}$.

Now our strategy is to increase the odd-indexed ranks $r_{b},r_{b+2}, \dots r_{2k+1}$ by $1$ and decrease the even-indexed ranks $r_{b+1}, r_{b+3}, \dots, r_{2k}$ by $1$. This preserves all the equalities $r_i + r_{i+1}=m$ for $i=b, b+1, \dots, 2k$. Now let us check what happens at the endpoints of the interval. By minimality of $b$, we have that $r_{b-1} + r_b < m$, so we can increase $r_b$ by $1$ and still be in the allowed region. Finally, by the assumption that $2k+1 \notin S$, we have that $r_{2k+1} + r_{2k+2} < m$, so we can increase $r_{2k+1}$ and still be in the allowed region. 

Finally, we can check how $d$ changes when $r_{b},r_{b+2}, \dots r_{2k+1}$ are increased by $1$ and $r_{b+1}, r_{b+3}, \dots, r_{2k}$ are decreased by $1$. If $\vec{r}'$ is the new vector of ranks, then a tedious but straightforward calculation yields that
\[
\begin{array}{lll}
d(\vec{a},\vec{r}')-d(\vec{a},\vec{r})&=&\ds\sum_{i=1}^{n}r'_i(2m-r'_{i-1}-r'_{i})-\sum_{i=1}^{n}r_i(2m-r_{i-1}-r_{i})\\\\
&=& \ds\sum_{i=b}^{2k+2}r'_i(2m-r'_{i-1}-r'_{i})-r_i(2m-r_{i-1}-r_{i})\\\\
&=& (r_b+1)(2m-r_{b-1}-r_b-1)+r_{2k+2}(2m-r_{2k+1}-1-r_{2k+2})\\\\
&+&\ds\sum_{i=b+1}^{2k+1}(r_i+(-1)^{i+1})(2m-r_{i-1}-r_{i})\\\\
&-&\ds[r_b(2m-r_{b-1}-r_b)+r_{2k+2}(2m-r_{2k+1}-r_{2k+2})\\\\
&+&\ds\sum_{i=b+1}^{2k+1}r_i(2m-r_{i-1}-r_i)]\\\\
&=&2m-2r_b-r_{b-1}-1-r_b+r_{2k+1}-r_{2k+2},
\end{array}
\]
now since we have $r=r_b, r_{b+2},\cdots, r_{2k+1},$ we get 
\[ d(\vec{r}') - d(\vec{r}) = 2m - (2r+1) - r_{b-1} - r_{2k+2}.\]
By minimality of $b$, we have that $r_{b-1} + r \le m-1$. On the other hand, by the assumption that $2k+1 \notin S$ we have that $r+r_{2k+2} \le m-1$. So $d(\vec{r}') - d(\vec{r}) \ge 1$. It can not be that $d$ was maximized at $\vec{r}$.

So we have proved that $2k+1 \in S$ for every $k=0,1,  \dots, (n/2) -1$. That is, $r_i + r_{i+1} = m$ for every odd $i$. Summing all these equalities gives $\sum_{i=1}^m r_i= nm / 2$, as desired. It also follows from $r_i + r_{i+1} = m$ and $\dim a_i = m$ that $\beta_{i}=0$ for $i=1, 3, \dots, n-1$.
\end{proof}


\begin{theorem}\label{prop:Betti for even n}
For a random chain complex
$0\leftarrow A_0\leftarrow A_1\leftarrow\cdots\leftarrow A_n\leftarrow 0,$ where $A_i$ is a vector space over $\mathbb{R}$ of dimension $a_i=m$, for some positive integer $m$, and for all $i$, where $n$ is even. Then almost surely $\beta_{2k}=\left\lfloor\frac{m}{n/2+1}\right\rfloor$ or $\beta_{2k}=\left\lceil\frac{m}{n/2+1}\right\rceil$ for $k=0,1,2,\cdots,n/2.$    
\end{theorem}
 
\begin{proof}
Recall that the dimension of the space of chain complexes is given by 
\begin{align*}
d(\vec{a},\vec{r})&=\sum_{i=1}^{n}r_i(2m-r_{i-1}-r_{i})\\
&=2m\sum_{i=1}^{n}r_i-\sum_{i=1}^{n}r_i(r_{i-1}+r_{i}).
\end{align*}
and we have seen in the proof of Theorem \ref{thm:equalodd} that $\sum_{i=1}^{n}r_i=\frac{nm}{2},$ so we have
$$d(\vec{a},\vec{r})=nm^2-\sum_{i=1}^{n}r_i(r_{i-1}+r_{i}).$$
Since we are taking $n$ and $m$ as given constants, the rank vector $\vec{r}=(r_1,r_2,...,r_n)$ maximizes the dimension $d$ if and only if it minimizes $f(\vec{r})=\sum_{i=1}^{n}r_i(r_{i-1}+r_{i})$. Now let $M$ be the set of extremal rank vectors $M=\{\vec{r}~|~\vec{r} \text{ minimizes } f \}$
and consider the function $g=\sum_{i=0}^{n}\beta_i^2.$ From Theorem \ref{thm:equaleven} we have $\beta_{2k-1}=0$ for all $k=1,2,\cdots,\frac{n}{2}.$ Then we can write $g$ as a function of $\vec{r}$ as $$g(\vec{r})=(m-r_1)^2+(m-r_2-r_3)^2+\cdots+(m-r_{n-2}-r_{n-1})^2+(m-r_n)^2.$$ Now let $M^{'}=\{\vec{r}~|~\vec{r} \text{ minimizes } g \}$, our claim is that $M=M^{'}$. Indeed, from the previous proof we know that $r_i+r_{i+1}=m$ for all odd $i$. Then
$$
\begin{array}{lll}
g(\vec{r})&=&(m-(m-r_2))^2-(m-r_2-(m-r_4))^2+\cdots+(m-r_n)^2\\
&=& r_2^2+(r_4-r_2)^2+\cdots+(r_n-r_{n-2})^2+(m-r_n)^2\\
&=&2r_2^2+2r_4^2+\cdots+2r_n^2-2r_2r_4-2r_4r_6-\cdots-2r_{n-2}r_n-2mr_n+m^2\\
&=& 2(f(\vec{r})-m\sum_{i=1}^nr_i)+m^2\\
&=& 2f(\vec{r})-nm^2+m^2.
\end{array}
$$
Therefore, $\vec{r}$ maximizes the dimension function $d$ if and only if it minimizes the function $g$. The lower bound is achieved when the Betti numbers are as close together as possible, given what we have already proved in Theorem \ref{thm:equaleven}, namely that 
\[ \sum_{i=0}^n \beta_i = m.\] So either $\beta_{2k}=\left\lfloor\frac{m}{n/2+1}\right\rfloor$ or $\beta_{2k}=\left\lceil\frac{m}{n/2+1}\right\rceil$ for $k=0,1,2,\cdots,\frac{n}{2}$. It is also straightforward to check that whenever the $\beta_{2k}$ only take these two values and such that $\sum_{i=0}^n \beta_i = m$, the dimension $d$ is the same. So each of these possibilities occurs with positive probability.
\end{proof}

This can give many possibilities for the vector $(\beta_0, \beta_1, \dots, \beta_n)$. Consider the following example. Just for convenience, let $n$ be a multiple of $4$, and then set $m = n^2 / 4 + n/4$. By the above, the rank vectors that occur with positive probability all yield $\beta_0 + \beta_2 + \dots + \beta_n = m$, and for each $i$ we have $\beta_{2i} = n/2 -1$ or $\beta_{2i} = n/2$. One easily checks that the number of $i$ such that $\beta_{2i} = n/2 -1$ is $n/4$.

So there are \[ \binom{n/2+1}{n/4} = \left( \sqrt{2} - o(1) \right)^n\]
possibilities for the vector of Betti numbers. This example shows that if $m$ is allowed to grow with $n$ then the number of positive-probability Betti-number vectors can be exponentially large in $n$.  

\section{Questions}

\begin{itemize}
    \item We conjecture that the analogue of Theorem \ref{thm:length3} holds for random chain complexes of arbitrary length. 

\begin{conjecture}
For a random chain complex of real vector spaces of length $n$, 
\[ 0 \leftarrow A_0 \xleftarrow{d_1} A_1 \xleftarrow{d_2}\cdots\xleftarrow{d_{n}} A_n \leftarrow 0\]
where $a_i+a_{i+2}\geq a_{i+1}$ for every $i$, almost surely $\sum \beta_i=|\chi|$.\\
\end{conjecture}

\item We have seen here examples of two qualitatively different types of random chain complex, one where the Betti-number vector concentrates on a single vector and another where it is spread out over exponentially many vectors. We wonder for which points in parameter space $\vec{a} \in \Z_{\ge 0}^{n+1}$ one gets each type of behavior. \\

\item It would be nice to have an efficient algorithm that, given $\vec{a}$, maximizes $d(\vec{a},\vec{r})$ over all feasible rank vectors $\vec{r}$. This would facilitate experimental work on random chain complexes. As we saw in an earlier example, for some $\vec{a}$ there are many $\vec{r}$ that maximize $d$.  What is the computational complexity? The example above shows that we cannot expect better than exponential time in $n$. But what if the Betti numbers are concentrated on a single vector --- can we find it efficiently?\\

\item Here we are concerned with showing that the Betti-number vector of a random chain complex is typically small in $L_1$ norm. It seems natural to study other norms. The proof of Theorem \ref{thm:equalodd} shows along the way that the Betti-number vector is also almost surely as small as possible in $L_2$ norm.\\

\item One could also study random chain complexes of free abelian groups, with random group homomorphisms given by random integer matrices. There might be torsion. One case of this has already been studied extensively. The zeroth homology $H_0$ in a length $1$ random complex is equivalent to studying the cokernel of a discrete random matrix. In the case of a square matrix, homology with $\Z$ coefficients is typically a finite abelian group, and its distribution is understood in terms of \emph{Cohen--Lenstra heuristics}. See the work on cokernals of random integer matrices by Wood \cite{Wood17,Wood23} and Nguyen--Wood \cite{NW22,NW25}, and also experimental work on torsion in homology of random simplicial complexes by Kahle, Lutz, Newman, and Parsons \cite{KLNP20}. \\
\end{itemize}

The second author dedicates this paper to the memory of Frank Lutz.

\bibliographystyle{alpha}
\bibliography{bibfile}

\end{document}